\documentclass[11pt]{amsart}

\usepackage{amssymb}

\newtheorem{theorem}{Theorem}[section]
\theoremstyle{plain}

\newtheorem{corollary}[theorem]{Corollary}

\newtheorem{lemma}[theorem]{Lemma}

\newtheorem{problem}[theorem]{Problem}
\newtheorem{proposition}[theorem]{Proposition}

\numberwithin{equation}{section}

\def\inn{\operatorname{Inn}}
\def\mlt{\operatorname{Mlt}}
\def\aut{\operatorname{Aut}}
\def\endom{\operatorname{End}}
\newcommand{\ext}[2]{\rtimes_{#1}^{#2}}

\title[Loops with commuting inner mappings]{Explicit constructions of loops with commuting inner mappings}

\author{Ale\v{s} Dr\'apal}
\email[Dr\'apal]{drapal@karlin.mff.cuni.cz}
\address[Dr\'apal]{Department of Algebra, Charles University, Sokolovsk\'a 83, 186 75 Prague,
Czech Republic}

\author{Petr Vojt\v{e}chovsk\'y}
\email[Vojt\v{e}chovsk\'y]{petr@math.du.edu}
\address[Vojt\v{e}chovsk\'y]{Department of Mathematics, University of Denver, 2360 S Gaylord St,
Denver, Colorado 80208, USA}

\thanks{The paper was written while the first author was
visiting University of Wisconsin-Madison as a Fulbright research scholar. He
was also supported by institutional grant MSM 0021620839.}



\begin{document}

\begin{abstract}
In 2004, Cs\"{o}rg\H{o} constructed a loop of nilpotency class three with
abelian group of inner mappings. Until now, no other examples were known. We
construct many such loops from groups of nilpotency class two by replacing the
product $xy$ with $xyh$ in certain positions, where $h$ is a central
involution. The location of the replacements is ultimately governed by a
symmetric trilinear alternating form.
\end{abstract}

\maketitle

\section{Introduction}

As is well known, a group is of nilpotency class at most two if and only if its
inner automorphism group is abelian. In $1946$, Bruck published a long paper
\cite{BruckTAMS} that influenced the development of loop theory for decades, in
which he observed that a loop of nilpotency class two possesses an abelian
inner mapping group. This paper is concerned with the converse of Bruck's
result.

While working on this problem in the early nineties, Kepka and Niemenmaa
\cite{NiemenmaaKepkaJAlg}, \cite{NiemenmaaKepka} proved that a finite loop with
abelian inner mapping group must be nilpotent. (Kepka later improved upon this
result and showed that if the inner mapping group is abelian and finite, then
the loop is nilpotent \cite{Kepka}.) But they did not establish an upper bound
on the nilpotency class of the loop, and, indeed, no such bound is presently
known.

Many experts believed that the converse of Bruck's result holds, just as in the
associative case. But in $2004$ (the result was published in $2007$),
Cs\"{o}rg\H{o} \cite{Csorgo} constructed a counterexample---a loop of
nilpotency class three with an abelian group of inner mappings. Throughout the
paper, we denote her loop by $C$.

Like Kepka and Niemenmaa, Cs\"{o}rg\H{o} has been using the technique of
$H$-connected transversals of groups, and her counterexample is therefore fully
embedded in group theory. She constructed a group $G$ of order $8192$ with a
subgroup $H$ of order $64$ and a two-sided transversal $A$ for $H$ in $G$ on
which one can define a loop (of order $128$) by $aH \cdot bH = cH$ if and only
if $abH = cH$. As is acknowledged in \cite{Csorgo}, her goal was to prove the
converse of Bruck's result, and she was gradually accumulating properties of a
minimal counterexample so that its existence could be refuted. But, in a twist
of events, she ended up constructing a counterexample.

Unfortunately, her approach did not lead to a general theory or construction
method that would allow one to obtain additional examples, much less to decide
how rare such examples are. Furthermore, the method of implicit construction by
means of transversals makes even the task of explicitly formulating the loop
operation somewhat nontrivial.

In this paper we deprive Cs\"{o}rg\H{o}'s example of its solitary nature. We
construct $C$ in two different, explicit ways, and, more importantly, we show a
general method that yields many other similar loops. The structure of the paper
corresponds quite closely to the history of our investigation, and we shall now
briefly describe both.

Using the GAP \cite{GAP} package LOOPS \cite{LOOPS}, we have constructed the
multiplication table of $C$, based on the description in \cite{Csorgo}, and
determined the sizes of the nuclei and of the associator subloop. We were quite
surprised that the latter consists of only two elements. This means that each
nontrivial associator in $C$ is equal to a central involution $h$. This led to
an early conjecture that $C$ can be obtained by the method of group table
modifications, as used in our earlier work \cite{DrapalEJC}, \cite{Drapal},
\cite{DrapalVojtechovsky}, \cite{Vojtechovsky}. More precisely, we conjectured
that there exists a group---we shall denote it again by $G$---such that
$C=(G,*)$, where $x*y \in \{xy, xyh\}$, for a fixed central involution $h \in
G$.

We were now facing two tasks: to determine $G$, and to identify those pairs
$(x,y)$, where the group operation should be modified. It is to be expected
that the modification is performed in a blockwise fashion, with $(x',y')$ and
$(x,y)$ behaving in the same way whenever $x'K=xK$ and $y'K=yK$, for a (large)
normal subgroup $K$ of $G$.

Our first reconstruction of $C$ has been obtained by different means, though.
In Section \ref{Sc:Extensions}, we develop the theory of nuclear extensions for
loops, which allows us to give an explicit formula for $C$. Since we were not
able to guess $G$ from this formula, we applied a greedy algorithm (whose
purpose was to maximize the number of associating triples) that resulted in
another loop $\overline C$ with similar properties. The explicit formula for
$\overline C$ is simple enough to connect it with a group $G$, which we
demonstrate in Section \ref{Sc:Crosshomomorphisms}.

By studying this single example $\overline C$, we developed a theory described
in Section \ref{Sc:Modifications}. The construction starts with a group $G$ of
nilpotency class at most three and produces a loop of nilpotency class three
with abelian inner mapping group. At first we thought that the process cannot
work if the starting group $G$ is of nilpotency class two, and when we tried to
refute this possibility we obtained some theorems that connect these groups to
triadditive (or trilinear) mappings via the iterated commutator $[[-,-],-]$.

It turns out that the sought loops can be, in fact, obtained from groups $G$ of
nilpotency class two, cf. Section \ref{Sc:NilpClassTwo}, and we were able to
reconstruct the loop $C$ in this way. In hindsight, our inability to do so in
the first place got a natural explanation: the subgroup $K$ that is used for
blockwise modifications is of order $2$.

The general construction (of obtaining loops from groups of nilpotency class
two) has three steps. The first step is strictly governed by the associated
group and the triadditive mapping. However, the second and the third steps
depend on many free parameters, which results in a combinatorial explosion.
Consequently, there are myriads (we do not know precisely how many) of loops of
order $128$, of nilpotency class three and with an abelian inner mapping group.
The general construction is given in Section \ref{Sc:NilpClassTwo}, and
explicit examples are calculated in Section \ref{Sc:Explicit}.

We conclude the paper with a list of open problems.

The paper relies heavily on machine computation, and all results not justified
by theory have been checked computationally. The GAP code used here can be
downloaded at \texttt{http://www.math.du.edu/\~{}petr} in section Publications.

In the planned sequel, we shall explain why our method cannot work for orders
less than $128$, and why it cannot work for loops of odd order.

\subsection{Background on loops}

A groupoid $Q$ is a \emph{loop} if the equations $ax=b$, $ya=b$ have unique
solutions $x$, $y\in Q$ whenever $a$, $b\in Q$ are given, and if there is $1\in
Q$, the \emph{neutral element} of $Q$, such that $a1=a=1a$ for every $a\in Q$.
A nonempty subset $S$ of a loop $Q$ is a \emph{subloop}, $S\le Q$, if $1\in S$
and $S$ is a loop with respect to the multiplication inherited from $Q$. We say
that $S\le Q$ is \emph{normal} in $Q$, $S\unlhd Q$, if $xS=Sx$, $x(yS)=(xy)S$,
$S(xy)=(Sx)y$ for every $x$, $y\in Q$.

Every element $x$ of a loop $Q$ gives rise to two permutations of $Q$, the
\emph{left translation} $L_x:y\mapsto xy$, and the \emph{right translation}
$R_x:y\mapsto yx$. The \emph{multiplication group} $\mlt Q$ of $Q$ is the group
generated by $\{L_x$, $R_x;\;x\in Q\}$. The mappings
\begin{displaymath}
    L(x,y) = L_{yx}^{-1}L_yL_x,\quad R(x,y) = R_{xy}^{-1}R_yR_x,\quad
    T(x) = R_x^{-1}L_x
\end{displaymath}
are known as \emph{left}, \emph{right}, and \emph{middle inner mappings},
respectively, and the \emph{inner mapping group} $\inn Q$ of $Q$ is the group
generated by all inner mappings of $Q$. In a complete analogy with groups, a
subloop $S$ of $Q$ is normal in $Q$ if and only if $\varphi(S)=S$ for every
$\varphi\in\inn Q$.

The \emph{commutator} of $x$, $y\in Q$ is defined by $xy = yx\cdot[x,y]$, and
the \emph{associator} of $x$, $y$, $z\in Q$ is defined by $(xy)z = x(yz)\cdot
[x,y,z]$. The \emph{associator subloop} of $Q$ is the smallest normal subloop
$A(Q)$ of $Q$ such that $Q/A(Q)$ is a group. In particular, $[x,y,z]\in A(Q)$
for every $x$, $y$, $z\in Q$.

The \emph{left nucleus} $N_\lambda(Q)$ of a loop $Q$ consists of all elements
$x\in Q$ such that $[x,y,z]=1$ for every $y$, $z\in Q$. Similarly, we have the
\emph{middle nucleus} $N_\mu(Q) = \{x\in Q;\; [y,x,z]=1$ for every $y$, $z\in
Q\}$, the \emph{right nucleus} $N_\rho(Q)=\{x\in Q;\; [y,z,x]=1$ for every $y$,
$z\in Q\}$, and the \emph{nucleus} $N(Q) = N_\lambda(Q)\cap N_\mu(Q)\cap
N_\rho(Q)$. All nuclei are associative but not necessarily normal subloops of
$Q$. The \emph{center} $Z(Q)$ of $Q$ consists of all elements $x\in N(Q)$ such
that $[x,y]=1$ for every $y\in Q$. It is then clear that $\varphi(Z(Q))=Z(Q)$
for every $\varphi\in\inn Q$, and hence $Z(Q)\unlhd Q$.

When $S\unlhd Q$, the \emph{factor loop} $Q/S$ is defined in the usual way.
Given $Q=Q_0$, let $Q_{i+1} = Q_i/Z(Q_i)$. If there is $m\ge 0$ such that $Q_m$
is trivial, we say that $Q$ is \emph{(centrally) nilpotent}, and if $m\ge 0$ is
the least integer for which $Q_m$ is trivial, we say that $Q$ is of
\emph{nilpotency class $m$}.

\section{Nuclear extensions}\label{Sc:Extensions}

Let $Q$, $K$, $F$ be loops. Then $Q$ is an \emph{extension} of $K$ by $F$ if
$K\unlhd Q$ and $Q/K\cong F$. Let us call an extension $Q$ of $K$ by $F$
\emph{nuclear} if $K$ is an abelian group such that $K\le N(Q)$.

In this subsection we generalize group extensions by abelian groups to nuclear
extensions of loops. We will need the following definitions:

A map $\theta:F\times F\to K$ is a \emph{cocycle} if
$\theta(x,1)=\theta(1,x)=1$ for every $x\in F$. Given a cocycle $\theta:F\times
F\to K$ and a homomorphism $\varphi:F\to\aut K$, $x\mapsto\varphi_x$, let
$K\ext{\theta}{\varphi}F$ be the groupoid $(K\times F,\circ)$ defined by
\begin{displaymath}
    (a,x)\circ(b,y) = (a\varphi_x(b)\theta(x,y),\,xy).
\end{displaymath}

Here is the key observation (we write $T_x$ for the inner mapping $T(x)$):

\begin{lemma}[Leong, Theorem 3 of \cite{Leong}]\label{Lm:Leong}
Let $Q$ be a loop with a normal subloop $K \leq N(Q)$. For each $x \in Q$,
define $\varphi_x = T_x |_K$. Then $\varphi_x \in \aut K$, and the mapping
$\varphi: Q\to \aut K$, $x\mapsto \varphi_x$ is a homomorphism.
\end{lemma}
\begin{proof}
First fix $a, b\in K$ and $x\in Q$. Since $K$ is normal in $Q$, we have
$\varphi_x(K)\le K\le N(Q)$. In particular, $T_x(ab)\cdot x = x\cdot ab =
xa\cdot b = (T_x(a)\cdot x)b = T_x(a)\cdot xb = T_x(a) (T_x(b)\cdot x) = T_x(a)
T_x(b) \cdot x$. Canceling $x$ on the right then shows that $\varphi_x$ is an
automorphism of $K$.

Now fix $a\in K$ and $x,y\in Q$. Let $z = T_{xy}(a)$. By the first part, $z\in
K$, and so $(zx)y = z(xy) = (xy)a = x(ya) = x (T_y(a)\cdot y) = x T_y(a)\cdot
y$. Upon canceling $y$ on the right, we get $zx = x T_y(a)$, i.e., $z = T_x T_y
(a)$. Hence $\varphi_{xy}=\varphi_x\varphi_y$, as claimed.
\end{proof}

\begin{theorem}[Nuclear extensions of loops]\label{Th:NuclearExtensions}
Let $K$ be an abelian group and $Q$,
$F$ loops. Then the following conditions are equivalent:
\begin{enumerate}
\item[(i)] $Q$ is an extension of $K$ by $F$ and $K\le N(Q)$,
\item[(ii)] $Q = K\ext{\theta}{\varphi}F$ for a cocycle $\theta:F\times F\to K$
and a homomorphism $\varphi:F\to\aut K$.
\end{enumerate}
\end{theorem}
\begin{proof}
Let $Q$ be an extension of $K$ by $F$, $K\le N(Q)$. Let $\pi:Q\to Q/K = F$ be
the natural projection, and let $\ell:F\to Q$ be such that $\ell(1)=1$ and
$\pi(\ell(x)) = x$ for every $x\in F$. Define $\varphi:F\to\aut K$,
$x\mapsto\varphi_x$, by $\varphi_x = T_{\ell(x)}|_K$. By Lemma \ref{Lm:Leong},
$\varphi_x\in\aut K$ for every $x\in F$. We have $\pi(\ell(xy)) = xy =
\pi(\ell(x))\pi(\ell(y)) = \pi(\ell(x)\ell(y))$, and thus for every $x$, $y\in
F$ there is a unique $\theta(x,y)\in K$ such that $\ell(x)\ell(y) =
\theta(x,y)\ell(xy)$.

Since $\ell(1)=1$, we have $\theta(x,1)=\theta(1,y)=1$, and $\theta:F\times
F\to K$ is a cocycle. Now, $\varphi_x\varphi_y = T_{\ell(x)}|_K\cdot
T_{\ell(y)}|_K = T_{\ell(x)\ell(y)}|_K = T_{\theta(x,y)\ell(xy)}|_K =
T_{\theta(x,y)}|_K\cdot \varphi_{xy}$, by Lemma \ref{Lm:Leong}. Since $K$ is
commutative, $T_k|_K=1$ for every $k\in K$. Hence $\varphi_{xy} =
\varphi_x\varphi_y$, and $\varphi:F\to\aut K$ is a homomorphism.

By the definition of $\ell$, for every $u\in Q$ there are uniquely determined
$a\in K$, $x\in F$ such that $u=a\ell(x)$. Define $\psi:Q\to
K\ext{\theta}{\varphi}F = (K\times F,\circ)$ by $\psi(u) = (a,x)$. It is clear
that $\psi$ is a bijection. Let $v=b\ell(y)$, with $b\in K$, $y\in F$. Then, on
the one hand, $\psi(u)\circ\psi(v) = (a,x)\circ (b,y) =
(a\varphi_x(b)\theta(x,y),xy)$. On the other hand, since $K\le N(Q)$, $uv =
a\ell(x)\cdot b\ell(y) = a(\ell(x)\cdot b\ell(y)) = a(\ell(x)b\cdot \ell(y)) =
a(T_{\ell(x)}(b)\ell(x)\cdot \ell(y)) = a(\varphi_x(b)\ell(x)\cdot \ell(y)) =
a(\varphi_x(b)\cdot \ell(x)\ell(y)) = a\varphi_x(b)\cdot \ell(x)\ell(y) =
a\varphi_x(b)\cdot \theta(x,y)\ell(xy) = a\varphi_x(b)\theta(x,y)\cdot
\ell(xy)$, and, consequently, $\psi(uv) = (a\varphi_x(b)\theta(x,y),xy)$. Thus
$\psi$ is an isomorphism.

Conversely, assume that $\theta:F\times F\to K$ is a cocycle, $\varphi:F\to\aut
K$ is a homomorphism, and $Q=K\ext{\theta}{\varphi}F$. Then $(1,1)\circ (b,y) =
(\varphi_1(b)\theta(1,y),y) = (b,y)$, and, similarly, $(a,x)(1,1) = (a,x)$,
showing that $(1,1)$ is the neutral element of $Q$. Note that $(a,x)\circ
(b,y)=(c,z)$ holds if and only if $a\varphi_x(b)\theta(x,y) = c$ and $xy=z$.
Hence, if $(a,x)$, $(c,z)$ are given, there is a unique $(b,y)$ satisfying
$(a,x)\circ(b,y)=(c,z)$, namely: $y$ is the unique solution to $xy=z$, and $b =
\varphi_x^{-1}(a^{-1}c\theta(x,y)^{-1})$. Similarly, there is a unique solution
$(a,x)$ when $(b,y)$, $(c,z)$ are given. Altogether, $Q$ is a loop.

We now show in detail that $K=(K,1)$ is a subloop of $N(Q)$. First,
$(a,1)\circ((b,y)\circ(c,z)) =(a,1)\circ(b\varphi_y(c)\theta(y,z),yz) =
(ab\varphi_y(c)\theta(y,z),yz)$, and $((a,1)\circ(b,y))\circ(c,z) =
(ab,y)\circ(c,z) = (ab\varphi_y(c)\theta(y,z),yz)$. Second,
$(b,y)\circ((a,1)\circ(c,z)) = (b,y)\circ (ac,z) =
(b\varphi_y(ac)\theta(y,z),yz)$, $((b,y)\circ(a,1))\circ(c,z) =
(b\varphi_y(a),y)\circ (c,z) = (b\varphi_y(a)\varphi_y(c)\theta(y,z),yz)$. As
$\varphi_y$ is a homomorphism, the two expressions coincide. Finally,
$(b,y)\circ ((c,z)\circ(a,1)) = (b,y)\circ (c\varphi_z(a),z) =
(b\varphi_y(c\varphi_z(a))\theta(y,z),yz)$, and $((b,y)\circ(c,z))\circ(a,1) =
(b\varphi_y(c)\theta(y,z),yz)\circ(a,1) =
(b\varphi_y(c)\theta(y,z)\varphi_{yz}(a),yz)$. As $\varphi_y$  and $\varphi$
are homomorphisms, the two expressions coincide.

We proceed to show that $K\unlhd Q$. Since $K\le N(Q)$, we get for free that
$K$ is closed under all left and right inner mappings of $Q$. It suffices to
show that $T_{(a,x)}(K)\subseteq K$ for every $a\in K$, $x\in F$. Now,
$T_{(a,x)}(b,1)$ belongs to $K$ if and only if there is $c\in K$ such that
$(a,x)\circ (b,1) = (c,1)\circ (a,x)$. Since $(a,x)\circ (b,1) =
(a\varphi_x(b),x)$ and $(c,1)\circ (a,x) = (ca,x)$, it suffices to take
$c=\varphi_x(b)$.

Finally, we must establish $Q/K\cong F$. But this is clear, since
$(a,x)\circ(b,y) = (a\varphi_x(b)\theta(x,y),xy)$ and $(1,x)\circ(1,y)$,
$(1,xy)$ coincide modulo $K$.
\end{proof}

\subsection{The first example}\label{Ss:FirstExample}

The loop $C$ from the Introduction has a normal nucleus isomorphic to the
elementary abelian group of order $16$ and such that $C/N(C)$ is an elementary
abelian group of order $8$. It is therefore a nuclear extension of $N(C)$ by
$C/N(C)$.

Note that once it is known that a loop $Q$ is a nuclear extension of $K\le
N(Q)$ by $Q/K$, the proof of Theorem \ref{Th:NuclearExtensions} is constructive
and provides the action $\varphi$ and the cocycle $\theta$, as soon as the
section mapping $\ell$ is chosen.

Hence, starting with a multiplication table for $C$ obtained from the original
construction of Cs\"{o}rg\H{o}, we can easily (with a computer) reconstruct $C$
as follows:

Let $K=\langle a_1, a_2, a_3, a_4\rangle$ be an elementary abelian group of
order $16$, and $F=\langle x_1, x_2, x_3\rangle$ an elementary abelian group of
order $8$. Set $a=a_1a_2a_3$ and $b=a_4$. Define a homomorphism
$\varphi:F\to\aut K$ by
\begin{displaymath}
    x_i\mapsto (a\leftrightarrow b,\, a_{i+1}\mapsto a_{i+1},\,a_{i+2}\mapsto a_{i+2}),
\end{displaymath}
where the addition in the subscripts is modulo $\{1,2,3\}$. Define a cocycle
$\theta: F\times F\to K$ by
\begin{displaymath}
\begin{array}{c|cccccccc}
                &1  &x_1        &x_2    &x_1x_2     &x_3    &x_1x_3     &x_2x_3     &x_1x_2x_3\\
    \hline
    1           &1  &1          &1      &1          &1      &1          &1          &1\\
    x_1         &1  &1          &1      &1          &a_2    &a_2        &aba_2      &aba_2\\
    x_2         &1  &a_3        &1      &a_3        &a_1    &aa_2       &a_1        &aa_2\\
    x_1x_2      &1  &a_3        &1      &a_3        &aa_3   &a          &a_3b       &b\\
    x_3         &1  &1          &1      &1          &1      &1          &1          &1\\
    x_1x_3      &1  &1          &1      &1          &a_2    &a_2        &aba_2      &aba_2\\
    x_2x_3      &1  &aba_3      &1      &aba_3      &a_1    &a_2b       &a_1        &a_2b\\
    x_1x_2x_3   &1  &aba_3      &1      &aba_3      &aa_3   &b          &a_3b       &a
\end{array}
\end{displaymath}
The resulting loop $K\ext{\theta}{\varphi}F$ is isomorphic to $C$.

Here are some properties of $C$: $N(C)=N_\rho(C)$ is elementary abelian of
order $16$, $|N_\lambda(C)|=|N_\mu(C)|=32$, $Z(C)=A(C)$, $|Z(C)|=2$. One can
interpret the fact that $|A(C)|=2$ as an indication that $C$ is very close to a
group, indeed.

\section{Extensions by crosshomomorphisms}\label{Sc:Crosshomomorphisms}

It is not clear how to deduce a general construction from a specific nuclear
extension, such as that of Subsection \ref{Ss:FirstExample}.

As far as nuclei are concerned, a more symmetric loop $\overline{C}$ is
obtained from $C$ by a simple greedy algorithm. We were able to develop the
general theory of Section \ref{Sc:Modifications} only after we understood the
loop $\overline{C}$, and we therefore devote considerable attention to it here.

Given a groupoid $Q$, let $\mu(Q) = |\{(a,b,c)\in Q\times Q\times Q;\;a(bc)\ne
(ab)c\}|$. Hence $\mu(Q)$ is a crude measure of (non)associativity of $Q$.

Let $T$ be a multiplication table of $C$ split into blocks of size $16\times
16$ according to the cosets modulo $N(C)$. Let $h$ be the unique nontrivial
central element of $C$.

(*) For $1<i < j\le 8$, let $T_{ij}$ be obtained from $T$ by multiplying the
$(i,j)$th block and the $(j,i)$th block of $T$ by $h$ on the right. Let $(s,t)$
be such that $\mu(T_{st})$ is minimal among all $\mu(T_{ij})$. If $\mu(T_{st})
\ge \mu(T)$, stop and return $T$. If $\mu(T_{st})<\mu(T)$, replace $T$ by
$T_{st}$, and repeat (*).

It turns out that the multiplication table $T$ found by the above greedy
algorithm yields another loop $\overline{C}$ of nilpotency class $3$ whose
inner mapping group is abelian. In addition, the following properties hold for
$\overline{C}$: $N(\overline{C})$ is elementary abelian of order $16$,
$|N_\lambda(\overline{C})| = |N_\mu(\overline{C})| = |N_\rho(\overline{C})| =
64$, $Z(\overline{C})=A(\overline{C})$, $|Z(\overline{C})|=2$. In particular,
$\overline{C}$ is not isomorphic to $C$.

We are now going to construct $\overline{C}$ anew. First we construct a certain
group $\overline G$, using a cocycle $\theta$ based on a crosshomomorphism. The
loop $\overline{C}$ can then be obtained in two ways: upon using a slight
variation of $\theta$, or, equivalently, by replacing $xy$ in $\overline{G}$
with $xyh$ for certain pairs $(x,y)\in \overline{G}\times \overline{G}$, where
$h$ is a nontrivial central element of $\overline{G}$.

\subsection{Crosshomomorphisms}\label{Ss:Crosshomomorphisms}

Recall that if $(A,\cdot)$, $(B,+)$ are groups and $\varphi:A\to\aut B$ is an
action, then a mapping $\gamma:A\to B$ is a \emph{crosshomomorphism} if
$\gamma(xy) = \gamma(x) + \varphi_x\gamma(y)$.

Let $F_1$, $F_2$ be multiplicative groups, and $K$ an additive abelian group.
Let $\psi:F_1\to(\endom K, +)$, $x\mapsto \psi_x$, and $\varphi:F_2\to(\aut
K,\circ)$, $y\mapsto\varphi_y$ be homomorphisms. Assume further that $\psi$ and
$\varphi$ commute, i.e., $\psi_x\varphi_y = \varphi_y\psi_x$. Extend the action
$\varphi$ to $F=F_1\times F_2$ by $\varphi_{(x_1,x_2)} = \varphi_{x_2}$.

Let $\gamma:F_2\to K$ be a map satisfying $\gamma(1)=0$, and let
$\theta:F\times F\to K$ be defined by
\begin{equation}\label{Eq:FromCross}
    \theta((x_1,x_2),(y_1,y_2)) = \psi_{y_1}\gamma(x_2).
\end{equation}

\begin{lemma}\label{Lm:CrossGroup}
Let $F=F_1\times F_2$, $K$, $\psi$, $\varphi$, $\gamma$ and $\theta$ be as
above. Then $K\ext{\theta}{\varphi}F$ is a group if and only if $\gamma$ is a
crosshomomorphism.
\end{lemma}
\begin{proof}
Direct computation shows that $K\ext{\theta}{\varphi} F$ is a group if and only
if $\theta(x,y) + \theta(xy,z) = \varphi_x\theta(y,z) + \theta(x,yz)$. Now,
\begin{align*}
    \theta((x_1,x_2),(y_1,y_2)) + \theta((x_1y_1,x_2y_2),(z_1,z_2))
    = \psi_{y_1}\gamma(x_2) + \psi_{z_1}\gamma(x_2y_2),
\end{align*}
while
\begin{align*}
    &\varphi_{(x_1,x_2)}\theta((y_1,y_2),(z_1,z_2)) + \theta((x_1,x_2),(y_1z_1,y_2z_2))
    = \varphi_{x_2}\psi_{z_1}\gamma(y_2) + \psi_{y_1z_1}\gamma(x_2)\\
    &= \psi_{z_1}\varphi_{x_2}\gamma(y_2) + \psi_{y_1}\gamma(x_2) +
        \psi_{z_1}\gamma(x_2)
    = \psi_{z_1}(\varphi_{x_2}\gamma(y_2) + \gamma(x_2)) +
        \psi_{y_1}\gamma(x_2).
\end{align*}
The two expressions coincide if and only if
\begin{equation}\label{Eq:Coincide}
    \psi_{z_1}(\varphi_{x_2}\gamma(y_2) + \gamma(x_2)) = \psi_{z_1}\gamma(x_2y_2).
\end{equation}
If $\gamma$ is a crosshomomorphism then $\gamma(x_2y_2) = \gamma(x_2)+
\varphi_{x_2}\gamma(y_2)$, and \eqref{Eq:Coincide} holds. Conversely, if
\eqref{Eq:Coincide} holds, use $z_1=1$ and the fact that $\psi_1$ is the
identity on $K$ to conclude that $\gamma$ is a crosshomomorphism.
\end{proof}

\begin{lemma}\label{Lm:V4} Let $K = \{(a_0,a_1,a_2);\;0\le a_i\le 1\}$
be a three-dimensional vector space over the two-element field. Let $V_4 =
\{b_1^{c_1}b_2^{c_2};\;0\le c_i\le 1\}$ be the Klein group. Then
$\varphi:V_4\to \aut K$ defined by
\begin{displaymath}
    \varphi_{b_1^{c_1}b_2^{c_2}}(a_0,a_1,a_2) = (a_0+c_2a_1+c_1a_2,a_1,a_2)
\end{displaymath}
is a homomorphism, and $\gamma:V_4\to K$ defined by
\begin{displaymath}
    \gamma(b_1^{c_1}b_2^{c_2}) = (c_1+c_2+c_1c_2,c_1,c_2)
\end{displaymath}
is a crosshomomorphism.
\end{lemma}
\begin{proof}
It is easy to see that $V_4$ acts on $K$ via $\varphi$. It remains to check
that $\gamma(xy) = \gamma(x)+\varphi_x\gamma(y)$ for every $x$, $y\in V_4$. We
have
\begin{displaymath}
    \gamma(b_1^{c_1}b_2^{c_2}\cdot b_1^{d_1}b_2^{d_2}) =
    \gamma(b_1^{c_1+d_1}b_2^{c_2+d_2}) = (c_1+d_1+c_2+d_2+(c_1+d_1)(c_2+d_2),
    c_1+d_1, c_2+d_2),
\end{displaymath}
while
\begin{align*}
    &\gamma(b_1^{c_1}b_2^{c_2})
        + \varphi_{b_1^{c_1}b_2^{c_2}}\gamma(b_1^{d_1}b_2^{d_2})\\
    & = (c_1+c_2+c_1c_2,c_1,c_2) +
        \varphi_{b_1^{c_1}b_2^{c_2}}(d_1 + d_2 + d_1d_2, d_1, d_2)\\
    & = (c_1 + c_2 + c_1c_2, c_1, c_2)
        + (d_1+d_2+d_1d_2 + c_2d_1 +c_1d_2, d_1,d_2)\\
    & = (c_1+c_2+c_1c_2+d_1+d_2 + d_1d_2 + c_2d_1 + c_1d_2, c_1+d_1, c_2+d_2).
\end{align*}
\end{proof}

Let $F_2=\langle \rho$, $\sigma;\;\rho^4=\sigma^2=(\sigma\rho)^2=1\rangle$ be
the dihedral group of order $8$. With $f_1=\sigma$ and $f_2=\sigma\rho$, we
have $F_2=\langle f_1,f_2\rangle$, and every element of $F_2$ can be written
uniquely as $(f_1f_2)^{2_i}f_1^jf_2^k$, where $0\le i$, $j$, $k\le 1$. We have
$Z(F_2)=\{1,f_1f_2\}$, and the projection $\pi:F_2\to F_2/Z(F_2)\cong
V_4=\langle b_1,b_2\rangle$ is determined by $f_i\mapsto b_i$. Hence the action
$\varphi$ of $V_4$ on $K$ from Lemma \ref{Lm:V4} can be extended to an action
$\varphi$ of $F_2$ on $K$ via $\varphi_x = \varphi_{\pi(x)}$, and the
crosshomomorphism $\gamma:V_4\to K$ from Lemma \ref{Lm:V4} can be extended into
a crosshomomorphism $\gamma:F_2\to K$ by $\gamma(x) = \gamma(\pi(x))$.

Let $F_1=\{0,1\}$ be the two element field, and let $\psi:F_1\to \endom K$ be
the scalar multiplication. Extend $\varphi$ once again into an action of
$F=F_1\times F_2$ on $K$ by $\varphi_{(x_1,x_2)} = \varphi_{x_2}$. Then
$\psi_x\varphi_y = \varphi_y\psi_x$ (since $\psi_x$ is either the zero map or
the identity on $K$). Let us calculate the explicit formula for the cocycle
$\theta$ associated with $\gamma$ via \eqref{Eq:FromCross}:
\begin{multline*}
    \theta((\ell,(f_1f_2)^{2i}f_1^jf_2^k),\,(\ell',(f_1f_2)^{2i'}f_1^{j'}f_2^{k'}))
    = \psi_{\ell'}\gamma((f_1f_2)^{2i}f_1^{j}f_2^{k})
    = \psi_{\ell'}\gamma(b_1^{j}b_2^{k})\\
    = \psi_{\ell'}(j+k+jk,j,k)
    = (\ell'(j+k+jk), \ell' j, \ell' k).
\end{multline*}
By Lemma \ref{Lm:CrossGroup}, $\overline{G} = K\ext{\theta}{\varphi}F$ is a
group (of nilpotency class three).

\subsection{The loop $\overline{C}$}

Upon modifying the cocycle $\theta$ slightly, we obtain a copy of
$\overline{C}$ and other loops of nilpotency class three with commuting inner
mappings.

For instance, define $\theta':F\times F\to K$ by
\begin{displaymath}
    \theta'((\ell,(f_1f_2)^{2i}f_1^jf_2^k),\,(\ell',(f_1f_2)^{2i'}f_1^{j'}f_2^{k'}))
    = (\ell'(i+j+k+jk), \ell' j, \ell' k).
\end{displaymath}
Then $K\ext{\theta'}{\varphi}F$ is isomorphic to $\overline{C}$.

\begin{figure}
\begin{displaymath}
\begin{array}{c||c|c|c|c}
    F&(\ell',i')=(0,0)&(\ell',i')=(0,1)&(\ell',i')=(1,0)&(\ell',i')=(1,1)\\
    \hline\hline
    (\ell,i)=(0,0)&1&1&1&1\\
    \hline
    (\ell,i)=(0,1)&1&1&h&h\\
    \hline
    (\ell,i)=(1,0)&1&1&1&1\\
    \hline
    (\ell,i)=(1,1)&1&1&h&h
\end{array}
\end{displaymath}
\caption{Modifying the group $\overline{G}$ by $h=(1,0,0)\in K$ to obtain
$\overline{C}$}\label{Fg:Modification}
\end{figure}

Let $h = (1,0,0)\in K$. Upon comparing the cocycles $\theta$ and $\theta'$, it
is now easy to describe $\overline{C}$ as a modification of $\overline{G}$,
where the product $xy$ is replaced by $xyh$ on certain blocks modulo $K$, as
indicated in Figure \ref{Fg:Modification}. In the figure, elements of
$F=\overline{G}/K$ are labeled as above, and each cell represents a $4\times 4$
block in the multiplication table of $F$.

To better understand the relationship between $\overline{G}$ and $\overline{C}$
we have considered additional variations of $\theta$. For $t_i\in \{0,1\}$,
$0\le i\le 6$, and $t = \sum_{i=0}^6 t_i2^i$, let
\begin{multline*}
    \theta_t((\ell,(f_1f_2)^{2i}f_1^jf_2^k),\,(\ell',(f_1f_2)^{2i'}f_1^{j'}f_2^{k'}))\\
    = (\ell'( t_0i + t_1j + t_2ij + t_3k + t_4ik + t_5jk + t_6ijk),\, \ell' j,\, \ell' k),
\end{multline*}
Then $K\ext{\theta_t}{\varphi}F$ is a group if and only if $t\in \{32, 34, 40,
42\}$, and all groups obtained in this way are isomorphic to $\overline{G}$.
More importantly, $K\ext{\theta_t}{\varphi}F$ is a nonassociative loop of
nilpotency class three with commuting inner mappings if and only if $t\in \{1$,
$3$, $9$, $11$, $33$, $35$, $41$, $43\}$, and all these loops are isomorphic to
the loop $\overline{C}$.

\begin{figure}
\begin{displaymath}
\begin{array}{|l|}
\hline
\text{$\mathbb F_2=\{0,1\}$ $\dots$ two-element field}\\
\text{$K = (\mathbb F_2)^3$ $\dots$ normal subgroup of $Q$}\\
\text{$D_8 = \langle \sigma,\rho; \sigma^2 = \rho^4 = (\sigma\rho)^2=1\rangle$ $\dots$ dihedral group of order $8$}\\
\text{$F = \mathbb F_2\times D_8$ $\dots$ factor group $Q/K$}\\
\text{$\varphi:F\to\aut K$ $\dots$ action}\\
\text{\quad\quad\quad\quad$\varphi_{(\ell,\,\rho^{2i}\sigma^j(\sigma\rho)^k)}(a,b,c)
    =(a+kb+jc,\,b,\,c)$}\\
\text{$\theta:F\times F\to K$ $\dots$ cocycle}\\
\text{\quad\quad\quad\quad$\theta((\ell,\,\rho^{2i}\sigma^j(\sigma\rho)^k),\,
        (\ell',\,\rho^{2i'}\sigma^{j'}(\sigma\rho)^{k'})) = (\ell'i,\,\ell'j,\,\ell'k)$}\\
\text{$(Q,\circ) = K\ext{\theta}{\varphi}F$ $\dots$ loop of nilpotency class three with abelian $\inn Q$}\\
\text{\quad\quad\quad\quad$(a,x)\circ(b,y) = (a+\varphi_x(b) + \theta(x,y),\,xy)$}\\
\hline
\end{array}
\end{displaymath}
\caption{Construction of the loop $Q=\overline{C}$ as an extension of $K$ by
$F$}\label{Fg:overlineC}.
\end{figure}

Since the cocycle $\theta_1$ is especially easy to describe, we present the
construction of $\overline{C}$ from scratch in Figure \ref{Fg:overlineC}. This
is presently the shortest description of a loop of nilpotency class three with
an abelian group of inner mappings.

Finally, the cocycle
\begin{displaymath}
    \theta''((\ell,(f_1f_2)^{2i}f_1^jf_2^k),\,(\ell',(f_1f_2)^{2i'}f_1^{j'}f_2^{k'}))
    = (\ell'(i+(k-k')j), \ell' j, \ell' k)
\end{displaymath}
produces a loop of nilpotency class three with commuting inner mappings that is
isomorphic neither to $C$ nor to $\overline{C}$. Since we will eventually be
able to construct many such examples, we do not pursue extensions any further.

\subsection{A power-associative loop that is the union of its nuclei}

Allow us to digress in this subsection.

A loop is \emph{power-associative} if each of its elements generates a
subgroup. The loop $\overline{C}$ contains a nonassociative power-associative
subloop that is a union of its nuclei. Since we are not aware of such a loop
appearing in the literature, we construct it here:

Let $K=\langle a_1, a_2, a_3, a_4\rangle$ be an elementary abelian group of
order $16$, and $F=\langle x_1, x_2\rangle$ an elementary abelian group of
order $4$. As above, let $a=a_1a_2a_3$ and $b=a_4$. Define a homomorphism
$\varphi:F\to\aut K$ by
\begin{displaymath}
    x_i \mapsto (a\mapsto a,\,b\mapsto b,\,a_i\mapsto a_i,\,a_3\mapsto aba_3),
\end{displaymath}
and a cocycle $\theta: F\times F\to K$ by
\begin{displaymath}
\begin{array}{c|cccc}
                &1  &x_1        &x_2        &x_1x_2\\
    \hline
    1           &1  &1          &1          &1\\
    x_1         &1  &a_1        &a_2b       &a_3\\
    x_2         &1  &aba_2      &a_2        &1\\
    x_1x_2      &1  &aa_3       &a          &a_3.
\end{array}
\end{displaymath}
Then $Q = K\ext{\theta}{\varphi}F$ is a nonassociative power-associative loop
of order $64$ such that $N_\lambda(Q)\cup N_\mu(Q)\cup N_\rho(Q) = Q$,
$|N(Q)|=16$, $|N_\lambda(Q)|=|N_\rho(Q)|=|N_\mu(Q)|=32$.

The subloop $Q\le\overline{C}$ was first spotted by Michael K.~Kinyon.

\section{Group modifications}\label{Sc:Modifications}

The properties of all examples constructed so far were verified by direct
machine computation. To remedy the situation, we now develop a theory based on
group modifications that also yields loops of nilpotency class three with
commuting inner mappings, but which does not require any machine computation.

\subsection{Conditions that make $\inn Q$ abelian}\label{Ss:Conditions}

Our point of departure is based on the structural properties of the loops $C$
and $\overline{C}$.

For the rest of this section, let $G$ be a group, $Z\le K\le N\unlhd G$, where
$N$ is abelian, $G/N$ is abelian, $Z\le Z(G)$, $K\unlhd G$, and $N/K\le
Z(G/K)$. Furthermore, let $\mu:G/K\times G/K\to Z$ be a mapping satisfying
$\mu(xK,K)=1=\mu(K,xK)$ for every $x\in G$.

Write $\mu(x,y)$ instead of $\mu(xK,yK)$, and define a groupoid $Q=(G,*)$ by
\begin{equation}\label{Eq:Q}
    x*y = xy\mu(x,y).
\end{equation}

\begin{lemma}\label{Lm:IsLoop} $Q$ is a loop.
\end{lemma}
\begin{proof}
We have $x*1= x = 1*x$ since $1\in K$. Assume that $x*y = x*z$ for some $x$,
$y$, $z\in G$. Then $xy\mu(x,y) = xz\mu(x,z)$, hence
$z^{-1}y=\mu(x,z)\mu(x,y)^{-1}\in K$, and so $z=yk$ for some $k\in K$. Thus $
xy\mu(x,y) = x*y = x*z = xyk\mu(x,yk) = xyk\mu(x,y)$, so $k=1$ and $y=z$.
Similarly, if $y*x =  z*x$ then $y=z$.

Note that $\mu(x,y\mu(u,v)) = \mu(x,y)$. Then for $x$, $z\in G$, we have
$x*x^{-1}z\mu(x,x^{-1}z)^{-1} = z$, so $y=x^{-1}z\mu(x,x^{-1}z)^{-1}$ is the
unique solution to $x*y=z$. Similarly, given $y$, $z\in Q$,
$x=zy^{-1}\mu(zy^{-1},y)^{-1}$ is the unique solution to $x*y=z$.
\end{proof}

\begin{lemma}\label{Lm:ZInCenters}
$Z\le Z(G)\cap Z(Q)$, and $G/Z\cong Q/Z$ is a group.
\end{lemma}
\begin{proof}
Let $z\in Z\le Z(G)\cap K$, $x$, $y\in G$. Then
\begin{align*}
    &z*x = zx = xz = x*z,\\
    &z*(x*y) = z(x*y) = zxy\mu(x,y) = zx*y = (z*x)*y,\\
    &x*(z*y) = x*zy = xzy\mu(x,y) = xz*y = (x*z)*y,\\
    &x*(y*z) = x*(z*y) = (x*z)*y = (z*x)*y = z*(x*y) = (x*y)*z.
\end{align*}
Thus $Z\le Z(Q)$, and $G/Z\cong Q/Z$ follows by the definition \eqref{Eq:Q}.
\end{proof}

Since $Q/Z\cong G/Z$ is a group, we have $A(Q)\le Z$, and thus $A(Q)\le Z(Q)\le
N(Q)\unlhd Q$. Such a situation has some well-known general consequences. For
instance, the associator $[x,y,z]$ depends only upon classes modulo $N(Q)$, by
\cite[Lemma 4.2]{KKP}. We will use this property freely.

Denote by $L_x$, $R_x$ the translations by $x$ in $Q$, rather than in $G$. For
convenience, allow us to redefine inner mappings for $Q$ by
\begin{displaymath}
    L(x,y) = L_y^{-1}L_x^{-1}L_{x*y},\quad
    R(x,y) = R^{-1}_{x*y}R_yR_x,\quad
    T(x) = L_x^{-1}R_x.
\end{displaymath}

\begin{lemma}\label{Lm:LRCommute}
We have $L(x,y)z = z[x,y,z]$ and $R(x,y)z = z[z,x,y]$. The subgroup $\langle
L(x,y),R(x,y);\;x$, $y\in G\rangle$ is abelian.
\end{lemma}
\begin{proof}
$L(x,y)z = z*[x,y,z]$ is equivalent to $(x*y)*z = x*(y*(z*[x,y,z]))$, which
holds since $[x,y,z]\in Z\le Z(Q)$, by Lemma \ref{Lm:ZInCenters}. Similarly,
$R(x,y)z = z*[z,x,y]$ is equivalent to $(z*x)*y = (z*[z,x,y])*(x*y)$, which
holds for the same reason.

Then
\begin{equation}\label{Eq:TwoLeftInner}
    L(x,y)L(u,v)z = L(x,y)(z[u,v,z]) = z[u,v,z][x,y,z[u,v,z]] = z[u,v,z][x,y,z],
\end{equation}
and $L(u,v)L(x,y)z = z[x,y,z][u,v,z]$. Also,
\begin{equation}\label{Eq:TwoRightInner}
    R(x,y)R(u,v)z = R(x,y)(z[z,u,v]) = z[z,u,v][z[z,u,v],x,y] = z[z,u,v][z,x,y],
\end{equation}
and $R(u,v)R(x,y)z = z[z,x,y][z,u,v]$. Finally,
\begin{displaymath}
    L(x,y)R(u,v)z = L(x,y)(z[z,u,v]) = z[z,u,v][x,y,z[z,u,v]] =
    z[z,u,v][x,y,z],
\end{displaymath}
and
\begin{displaymath}
    R(u,v)L(x,y)z = R(u,v)(z[x,y,z]) = z[x,y,z][z[x,y,z],u,v] =
    z[x,y,z][z,u,v].
\end{displaymath}
\end{proof}

Let $y^x=x^{-1}yx$, and $[y,x]=y^{-1}y^x$. The following lemma shows how
conjugations differ in $G$ and $Q$:

\begin{lemma}\label{Lm:Conjugations}
$T(x)y = y^x\mu(y,x)\mu(x,y^x)^{-1}$.
\end{lemma}
\begin{proof}
Note that $L_a^{-1}(b)$ is the unique solution $y$ to the equation $a*y=b$.
Thus, as we have showed in the proof of Lemma \ref{Lm:IsLoop}, $L_a^{-1}(b) =
a^{-1}b\mu(a,a^{-1}b)^{-1}$. Then
\begin{displaymath}
    T(x)y = L_x^{-1}R_x(y) = x^{-1}R_x(y)\mu(x,x^{-1}R_x(y))^{-1},
\end{displaymath}
and we are done by $x^{-1}R_x(y) = x^{-1}(y*x) = x^{-1}yx\mu(y,x) =
y^x\mu(y,x)$.
\end{proof}

Define $\delta:G/K\times G/K\to Z$ by
\begin{displaymath}
    \delta(x,y) = \mu(x,y)\mu(y,x)^{-1}.
\end{displaymath}

Consider these conditions on $\mu$ and $\delta$, that we have observed while
attempting to construct $C$ and $\overline{C}$ by extensions:
\begin{align}
    &\mu(xy,z) = \mu(x,z)\mu(y,z)\text{ if $\{x,y,z\}\cap N\ne\emptyset$},\label{Eq:C1}\\
    &\mu(x,yz) = \mu(x,y)\mu(x,z)\text{ if $\{x,y,z\}\cap N\ne\emptyset$},\label{Eq:C2}\\
    &z^{yx}\delta([z,y],x) = z^{xy}\delta([z,x],y).\label{Eq:C3}
\end{align}

\begin{lemma}\label{Lm:TLRCommute}
Suppose that \eqref{Eq:C1}, \eqref{Eq:C2} hold. Then $N\le N(Q)$, and $T(x)$
commutes with $L(u,v)$, $R(u,v)$.
\end{lemma}
\begin{proof}
It is easy to see that
\begin{equation}\label{Eq:Associator}
    [x,y,z] = \mu(x,y)\mu(xy,z)\mu(x,yz)^{-1}\mu(y,z)^{-1}\in Z.
\end{equation}
Then $[x,y,z]=1$ whenever $\{x,y,z\}\cap N\ne\emptyset$, and so $N\le N(Q)$.

By Lemmas \ref{Lm:LRCommute} and \ref{Lm:Conjugations},
\begin{displaymath}
    L(u,v)T(x)y = (T(x)y)[u,v,T(x)y] = y^x \mu(y,x)\mu(x,y^x)^{-1}[u,v,y^x],
\end{displaymath}
and
\begin{align*}
    T(x)L(u,v)y &= T(x)(y[u,v,y]) = (y[u,v,y])^x
    \mu(y[u,v,y],x)\mu(x,(y[u,v,y])^x)^{-1}\\
        &= y^x[u,v,y]\mu(y,x)\mu(x,y^x)^{-1}.
\end{align*}
So the equality $L(u,v)T(x)y = T(x)L(u,v)y$ holds if and only if
$[u,v,y^x]=[u,v,y]$. This is true since $y^x=y[y,x]$, and $[y,x]\in G'\le N\le
N(Q)$.

Similarly
\begin{displaymath}
    R(u,v)T(x)y = (T(x)y)[T(x)y,u,v] = y^x \mu(y,x)\mu(x,y^x)^{-1}[y^x,u,v],
\end{displaymath}
and
\begin{align*}
    T(x)R(u,v)y &= T(x)(y[y,u,v]) = (y[y,u,v])^x
    \mu(y[y,u,v],x)\mu(x,(y[y,u,v])^x)^{-1}\\
        &= y^x[y,u,v]\mu(y,x)\mu(x,y^x)^{-1}.
\end{align*}
So the equality $R(u,v)T(x)y = T(x)R(u,v)y$ holds if and only if
$[y^x,u,v]=[y,u,v]$, and we finish as before.
\end{proof}

\begin{proposition}\label{Pr:InnAbelian}
Assume that \eqref{Eq:C1}, \eqref{Eq:C2} hold. Then $\inn Q$ is abelian if and
only if \eqref{Eq:C3} holds.
\end{proposition}
\begin{proof}
In view of Lemmas \ref{Lm:LRCommute} and \ref{Lm:TLRCommute}, it suffices to
show that $T(x)T(y) = T(y)T(x)$ for every $x$, $y$. Using Lemma
\ref{Lm:Conjugations}, a straightforward calculation yields
\begin{equation}\label{Eq:TwoMiddleInner}
    T(x)T(y)z =
    z(yxz)^{-1}(zyx)\mu(z^y,x)\mu(x,z^{yx})^{-1}\mu(z,y)\mu(y,z^y)^{-1}.
\end{equation}
Sine $G'\le N$, we have
\begin{align*}
    &\mu(z^y,x) = \mu(z[z,y],x) = \mu(z,x)\mu([z,y],x),\\
    &\mu(y,z^y) = \mu(y,z[z,y]) = \mu(y,z)\mu(y,[z,y]).
\end{align*}
In any group, $[z,xy] = [z,y][z,x][[z,x],y]$, and since $[[z,x],y]$ belongs to
$K$ (as $N/K\le Z(G/K)$), we have
\begin{displaymath}
    \mu(x,z^{yx}) = \mu(x,z[z,yx]) = \mu(x,z)\mu(x,[z,yx])
     = \mu(x,z)\mu(x,[z,y])\mu(x,[z,x]).
\end{displaymath}
Putting all these facts together, we can rewrite $T(x)T(y)z$ as
\begin{multline*}
    z\mu(z,x)\mu(x,z)^{-1}\mu(z,y)\mu(y,z)^{-1}\\
        (yxz)^{-1}(zyx)\mu([z,y],x)\mu(x,[z,y])^{-1}\\
        \mu(x,[z,x])^{-1}\mu(y,[z,y])^{-1}.
\end{multline*}
Upon interchanging $x$ and $y$, we deduce that $T(y)T(x)z$ is equal to
\begin{multline*}
    z\mu(z,x)\mu(x,z)^{-1}\mu(z,y)\mu(y,z)^{-1}\\
        (xyz)^{-1}(zxy)\mu([z,x],y)\mu(y,[z,x])^{-1}\\
        \mu(x,[z,x])^{-1}\mu(y,[z,y])^{-1}.
\end{multline*}
The result then follows.
\end{proof}

\subsection{Consequences of the conditions}\label{Ss:Consequences}

\begin{proposition}\label{Pr:ClassAtMost3}
Assume that \eqref{Eq:C3} holds. Then both $G$ and $Q$ are of nilpotency class
at most three.
\end{proposition}
\begin{proof}
By \eqref{Eq:C3}, $z^{-xy}z^{yx}=\delta([z,x],y)\delta([z,y],x)^{-1}\in Z \le
Z(G)$. Since $z^{-xy}z^{yx} = (xy)^{-1}[z,[y^{-1},x^{-1}]](xy)$ holds in any
group,
\begin{equation}\label{Eq:Aux1}
    [z,[y^{-1},x^{-1}]]=\delta([z,x],y)\delta([z,y],x)^{-1}=z^{-xy}z^{yx}
\end{equation}
follows.

Set $U = G'Z$ and consider the series $1 \le Z \le U \le G$. We have $Z \le
Z(G)\cap Z(Q)$ by Lemma \ref{Lm:ZInCenters}, and $U/Z \le Z(G/Z)$ by
\eqref{Eq:Aux1}. The latter inclusion holds for both operations, as $G/Z\cong
Q/Z$ by Lemma \ref{Lm:ZInCenters}.
\end{proof}

Although we originally believed that $Q$ of nilpotency class three cannot be
obtained from $G$ of nilpotency class two, it turns out that it can happen and
that ample examples exist. We therefore focus on the (simpler) case when $G$ is
of nilpotency class two and $Q$ is of nilpotency class three.

As an immediate consequence of \eqref{Eq:Aux1}, we have:

\begin{corollary}\label{Cr:GroupClass2} Assume that \eqref{Eq:C3} holds. Then
the following conditions are equivalent:
\begin{enumerate}
\item[(i)] $G$ is of nilpotency class at most two,
\item[(ii)] $z^{xy}=z^{yx}$ for every $x$, $y$, $z\in G$,
\item[(iii)] $\delta([z,x],y)=\delta([z,y],x)$ for every $x$, $y$, $z\in G$.
\end{enumerate}
\end{corollary}

In view of Proposition \ref{Pr:ClassAtMost3}, the following result is relevant.
It is a consequence of the Hall-Witt identity for groups
\begin{displaymath}
    [[x,y^{-1}],z]^y[[y,z^{-1}],x]^z[[z,x^{-1}],y]^x=1.
\end{displaymath}

\begin{lemma}\label{Lm:InGroupOfClass3}
Let $H$ be a group of nilpotency class at most three.
Then
\begin{displaymath}
    [x,[y,z]][y,[z,x]][z,[x,y]]=1,
\end{displaymath}
and $[x,[y,z]] =[x,[y^{-1},z^{-1}]]$ for every $x$, $y$, $z\in H$.
\end{lemma}

\begin{lemma}\label{Lm:Interchange12}
Assume that \eqref{Eq:C1}, \eqref{Eq:C2} hold. Then $\delta([z,x],y) =
\delta([x,z],y)^{-1}$.
\end{lemma}
\begin{proof}
First note that $\mu(zx,y)=\mu(xz[z,x],y) = \mu(xz,y)\mu([z,x],y)$. This means
that $\mu([z,x],y) = \mu(zx,y)\mu(xz,y)^{-1}$. Similarly, $\mu(y,[z,x]) =
\mu(y,zx)\mu(y,xz)^{-1}$. Hence $\delta([z,x],y) =
\mu([z,x],y)\mu(y,[z,x])^{-1} =
\mu(zx,y)\mu(xz,y)^{-1}\mu(y,zx)^{-1}\mu(y,xz)$. The equality $\delta([z,x],y)
= \delta([x,z],y)^{-1}$ follows.
\end{proof}

\begin{proposition}\label{Pr:LoopClass2} Assume that
\eqref{Eq:C1}--\eqref{Eq:C3} hold. Then the following conditions are
equivalent:
\begin{enumerate}
\item[(i)] $Q$ is of nilpotency class at most two,
\item[(ii)] $G'\le Z(Q)$,
\item[(iii)] $\delta([x,y],z)\delta([y,z],x)\delta([z,x],y)=1$.
\end{enumerate}
\end{proposition}
\begin{proof}
By Lemma \ref{Lm:ZInCenters}, $Z\le Z(Q)$ and $Q/Z\cong G/Z$. This means that
$Z(Q)$ is not only a subgroup of $Q$ but also a subgroup of $G$, and
$Q/Z(Q)\cong G/Z(Q)$. Now, $Q$ is of nilpotency class at most two if and only
if $Q/Z(Q)$ is abelian, which is the same as $G/Z(Q)$ being abelian, which is
equivalent to $G'\le Z(Q)$.

By Lemma \ref{Lm:TLRCommute}, $G'\le N\le N(Q)$. Thus the condition $G'\le
Z(Q)$ holds if and only if $z*[y,x]=[y,x]*z$ for every $x$, $y$, $z\in G$. This
is the same as $z[y,x]\mu(z,[y,x]) = [y,x]z\mu([y,x],z)$, or, equivalently,
\begin{displaymath}
    [z,[y^{-1},x^{-1}]] = [z,[y,x]] = \delta([y,x],z) = \delta([x,y],z)^{-1},
\end{displaymath}
by Lemmas \ref{Lm:InGroupOfClass3} and \ref{Lm:Interchange12}. Then
\eqref{Eq:Aux1} and Lemma \ref{Lm:Interchange12} yield
\begin{displaymath}
    \delta([x,y],z)^{-1} =
    [z,[y^{-1},x^{-1}]] = \delta([z,x],y)\delta([z,y],x)^{-1}
    = \delta([z,x],y)\delta([y,z],x),
\end{displaymath}
and we are done.
\end{proof}

In particular, if $G$ is abelian then $Q$ cannot be of nilpotency class three.

\begin{corollary}\label{Cr:23}
Assume that \eqref{Eq:C1}--\eqref{Eq:C3} hold. Then $Q$ is of nilpotency class
three and $G$ is of nilpotency class two if and only if $\delta([x,y],z) =
\delta([x,z],y)$ for every $x$, $y$, $z\in G$, and $\delta([x,y],z)\ne 1$ for
some $x$, $y$, $z\in G$.
\end{corollary}
\begin{proof}
By Corollary \ref{Cr:GroupClass2}, $\delta([x,y],z) = \delta([x,z],y)$ holds
for every $x$, $y$, $z\in G$ if and only if $G$ is of nilpotency class two, in
which case Lemma \ref{Lm:Interchange12} yields
\begin{equation}\label{Eq:Cancel}
    1 = \delta([y,z],x)\delta([y,z],x)^{-1}
    = \delta([y,z],x)\delta([z,y],x)
    = \delta([y,z],x)\delta([z,x],y).
\end{equation}
Using \eqref{Eq:Cancel} and Proposition \ref{Pr:LoopClass2}, $Q$ is then of
nilpotency class three if and only if there are $x$, $y$, $z\in G$ such that
$1\ne \delta([x,y],z)\delta([y,z],x)\delta([z,x],y) = \delta([x,y],z)$.
\end{proof}

When $A$, $B$ are groups then $f:A^3\to B$ is said to be \emph{symmetric
triadditive} if $f(a_1,a_2,a_3) = f(a_{\sigma(1)},a_{\sigma(2)},a_{\sigma(3)})$
for every permutation of $\sigma$ of $\{1,2,3\}$ and every $a_1$, $a_2$,
$a_3\in A$, and $f(ab,c,d) = f(a,c,d)f(b,c,d)$ for every $a$, $b$, $c$, $d\in
A$.

\begin{proposition}\label{Pr:SymmForm}
Assume that \eqref{Eq:C1}--\eqref{Eq:C3} holds, $G$ is of nilpotency class two
and $Q$ is of nilpotency class three. Then there exists a subgroup $A \le Z$ of
exponent two and a nontrivial symmetric triadditive mapping $f:(G/N)^3 \to A$
such that
\begin{displaymath}
    \delta([x,y],z) = f(xN,yN,zN)
\end{displaymath}
for all $x$, $y$, $z\in G$.
\end{proposition}
\begin{proof}
Let $f:G^3\to Z$ be defined by $f(x,y,z) = \delta([x,y],z)$. By Corollary
\ref{Cr:23}, $f$ is nontrivial, and $\delta([x,y],z) = \delta([x,z],y)$, so
$f(x,y,z)$ is invariant under the permutation $(2,3)$ of its arguments. By
\eqref{Eq:Cancel},
\begin{displaymath}
    f(x,y,z) = \delta([x,y],z) = \delta([x,y],z)\delta([y,z],x)\delta([z,x],y),
\end{displaymath}
which shows that $f(x,y,z)$ is invariant under the permutation $(1,2,3)$ of its
arguments. Altogether, $f$ is symmetric.

By Lemma \ref{Lm:InGroupOfClass3},
$1=[x,[z^{-1},y^{-1}]][z,[y^{-1},x^{-1}]][y,[x^{-1},z^{-1}]]$. By
\eqref{Eq:Aux1}, we can rewrite this as
\begin{displaymath}
    1=\delta([x,y],z)\delta([x,z],y)^{-1}
    \delta([z,x],y)\delta([z,y],x)^{-1}
    \delta([y,z],x)\delta([y,x],z)^{-1},
\end{displaymath}
and by Lemma \ref{Lm:Interchange12} and \eqref{Eq:Cancel} we can further
simplify it to
\begin{displaymath}
    1 = (\delta([x,y],z)\delta([y,z],x)\delta([z,x],y))^2 = \delta([x,y],z)^2.
\end{displaymath}
This guarantees the existence of the subgroup $A\le Z$ of exponent two.

In any group, $[xy,z] = [x,z][[x,z],y][y,z]$. Since $[x,z]$, $[y,z]\in N$ and
$[[x,z],y]\in K$, we have $\mu([xy,z],u) = \mu([x,z],u)\mu([y,z],u)$, and
$\mu(u,[xy,z]) = \mu(u,[x,z])\mu(u,[y,z])$. Then $\delta([xy,z],u) =
\delta([x,z],u)\delta([y,z],u)$, and $f$ is triadditive.

Let $n\in N$. Then $[x,n]\in K$ and $\delta([x,n],z)=1$. Thus, by additivity,
$\delta([x,yn],z) = \delta([x,y],z)\delta([x,n],z) = \delta([x,y],z)$ and that
is why the value $f(x,y,z)$ depends only upon classes modulo $N$.
\end{proof}

\section{Constructing loops from symmetric trilinear alternating forms}\label{Sc:NilpClassTwo}

As we have just shown, if $Q$ is a loop of nilpotency class three with
commuting inner mappings obtained as a modification of a group $G$ of
nilpotency class two by $\mu$, then $\delta$ gives rise to a nontrivial
symmetric triadditive form.

We now show a partial converse. Namely, that it is possible to construct $G$
and $\mu$ (and hence $Q$) with the desired properties from certain groups of
nilpotency class two.

Throughout this section, let $H$ be a group of nilpotency class two such that
$H'=Z(H)$, $H/H'$ is an elementary abelian 2-group with basis
$\{e_1H',\dots,e_dH'\}$, and $H'$ is an elementary abelian 2-group with basis
$\{[e_i,e_j];\;1\le i<j\le d\}$. In addition, let $A=\{1,-1\}$, and let
$f:(H/H')^3\to A$ be a symmetric trilinear alternating form (we can view $H/H'$
as a vector space over $A$). For $u$, $v$, $w\in H$, we write $f(u,v,w)$
instead of the formally more precise $f(uH',vH',wH')$.

Starting with $f$, we are going to construct $\delta:H\times H\to A$ and
$\mu:H\times H\to A$ so that $f(u,v,w) = \delta([u,v],w)$, $\delta(u,v) =
\mu(u,v)\mu(v,u)^{-1}$, and such that \eqref{Eq:C1}--\eqref{Eq:C3} hold for
$\mu$ and $\delta$, with $H'$ in place of $N$.

We can then set $G=A\times H$ (any extension of $A$ by $H$ would do) and use
the mappings $\mu$, $\delta:H\times H\to A$ to obtain the loop $Q=(G,*)$
according to \eqref{Eq:Q}. By Proposition \ref{Pr:InnAbelian} and Corollary
\ref{Cr:23}, $Q$ is then a loop of nilpotency class three with commuting inner
mappings, \emph{provided that $f$ is nontrivial}.

Let $M=H'$. The construction of $\delta$ and $\mu$ is in three steps. First,
the condition $f(u,v,w)=\delta([u,v],w)$ forces $\delta$ on $M\times H$.
Second, the extension of $\delta$ from $M\times H$ to $H\times H$ depends on
certain free parameters. Third, once $\delta:H\times H\to A$ is given,
additional free parameters are needed to obtain $\mu$.

\subsection{Constructing $\delta$}

For $1\le i$, $j$, $k\le d$ and $m\in M$ let
\begin{equation}\label{Eq:DeltaMH}
    \delta([e_i,e_j],e_km) = f(e_i,e_j,e_k),
\end{equation}
and extend $\delta$ linearly into a mapping $\delta:M\times H\to A$. Then
$\delta$ satisfies $\delta(m_1m_2,h) = \delta(m_1,h)\delta(m_2,h)$ and
$\delta(m,h_1h_2) = \delta(m,h_1)\delta(m,h_2)$ for every $m$, $m_1$, $m_2\in
M$ and $h$, $h_1$, $h_2\in H$. Also, $\delta(M,M)=1$, because $f$ vanishes
whenever one of its arguments is trivial.

Our present task is to construct $\delta:H\times H\to A$ such that
\begin{align}
    &\delta([u,v],w) = f(u,v,w)\text{ for every $u$, $v$, $w\in H$},\notag\\
    &\delta(u,v) = \delta(v,u)^{-1}\text{ for every $u$, $v\in H$},\label{Eq:WantDelta}\\
    &\delta(u,vw) = \delta(u,v)\delta(u,w)\text{ if $\{u,v,w\}\cap
    M\ne\emptyset$.}\notag
\end{align}
(If $\{u,v,w\}\cap M\ne\emptyset$, we then also have $\delta(uv,w) =
\delta(w,uv)^{-1} = \delta(w,u)^{-1}\delta(w,v)^{-1} =
\delta(u,w)\delta(v,w)$.)

\begin{lemma} $\delta:M\times H\to A$ satisfies
$\delta([u,v],w) = f(u,v,w)$ for every $u$, $v$, $w\in H$.
\end{lemma}
\begin{proof}
Let $u\in e_1^{u_1}\cdots e_d^{u_d}M$, and $v\in e_1^{v_1}\cdots e_d^{v_d}M$,
where $0\le u_i$, $v_i\le 1$. Since $H$ is of nilpotency class $2$, we have
$[xy,z]=[x,z][y,z]$ and $[x,yz] = [x,y][x,z]$ for every $x$, $y$, $z\in H$.
Thus
\begin{displaymath}
    \delta([u,v],w) = \prod_{i\ne j}\delta([e_i^{u_i},e_j^{v_j}],w),
\end{displaymath}
and, since $f$ is trilinear and alternating,
\begin{displaymath}
    f(u,v,w) = \prod_{i\ne j}f(e_i^{u_i},e_j^{v_j},w).
\end{displaymath}
It therefore suffices to show that $\delta([e_i^{u_i},e_j^{v_j}],w) =
f(e_i^{u_i},e_j^{v_j},w)$. This is clearly true when $u_i=0$ or $v_j=0$, and
when $u_i=v_j=1$, it follows from the definition \eqref{Eq:DeltaMH} of
$\delta$.
\end{proof}

We now extend $\delta:M\times H\to A$ to $\delta:H\times H\to A$. Let $T$ be a
transversal for $M$ in $H$, $T=\{t_1$, $\dots$, $t_k\}$, $t_1=1$. For $1\le i$,
$j\le k$, choose $\delta(t_i,t_j)\in A$ as follows:
\begin{align}
    &\delta(t_1,t_j) = 1 \text{ for every $1\le j\le k$},\notag\\
    &\delta(t_i,t_j) \text{ arbitrary when $1<i<j\le k$},\notag\\
    &\delta(t_j,t_i)=\delta(t_i,t_j)^{-1} \text{ when $1<i<j\le k$},\label{Eq:DeltaParams}\\
    &\delta(t_i,t_i)=1 \text{ for every $1\le i\le k$.}\notag
\end{align}

Every element $h\in H$ can be written uniquely as $h=mt$ for some $m\in M$,
$t\in T$, and we define $\delta:H\times H\to A$ by
\begin{equation}\label{Eq:DeltaMM}
    \delta(mt,m't') = \delta(m,t')\delta(m',t)^{-1}\delta(t,t'),
\end{equation}
where $\delta(m,t')$, $\delta(m',t)$ have already been defined above.

Note that the new definition \eqref{Eq:DeltaMM} gives $\delta(m,m't') =
\delta(m,t')\delta(m',1)\delta(t_1,t') = \delta(m,t')$, while the old
definition \eqref{Eq:DeltaMH} gives $\delta(m,m't') = \delta(m,m')\delta(m,t')
= \delta(m,t')$. Hence $\delta:H\times H\to A$ extends the map $\delta:M\times
H\to A$.

\begin{lemma}\label{Lm:delta2}
Let $\delta:H\times H\to A$ be defined as above. Then:
\begin{enumerate}
\item[(i)] $\delta(u,v) = \delta(v,u)^{-1}$ for every $u$, $v\in H$,
\item[(ii)] $\delta(u,mv) = \delta(u,m)\delta(u,v)$ for every $u$, $v\in H$,
$m\in M$,
\item[(iii)] $\delta(u,vm) = \delta(u,v)\delta(u,m)$ for every $u$, $v\in H$,
$m\in M$,
\item[(iv)] $\delta(m,uv) = \delta(m,u)\delta(m,v)$ for every $u$, $v\in H$,
$m\in M$.
\end{enumerate}
\end{lemma}
\begin{proof}
(i) We have $\delta(mt,m't') = \delta(m,t')\delta(m',t)^{-1}\delta(t,t')$, and
also $\delta(m't',mt)^{-1} = \delta(m',t)^{-1}\delta(m,t')\delta(t',t)^{-1}$.
Hence we are done by $\delta(t,t') = \delta(t',t)^{-1}$ of
\eqref{Eq:DeltaParams}.

(ii) Let $u=nt$, $v=n't'$, where $n$, $n'\in M$ and $t$, $t'\in T$. Then
$\delta(u,mv) = \delta(nt,(mn')t') = \delta(n,t')\delta(mn',t)^{-1}
\delta(t,t')$. On the other hand, $\delta(u,m)\delta(u,v)$ is equal to
$\delta(nt,m)\delta(nt,n't') = \delta(m,t)^{-1}\delta(n,t')\delta(n',t)^{-1}
\delta(t,t')$. Since $\delta(m,t)\delta(n',t) = \delta(mn',t)$, we are done.

Part (iii) is an immediate consequence of (ii), $M\subseteq Z(H)$, and the fact
that $A$ is abelian. We have already observed (iv).
\end{proof}

\subsection{Constructing $\mu$}

We now need a map $\mu:H\times H\to A$ such that
\begin{align}
    &\delta(u,v) = \mu(u,v)\mu(v,u)^{-1}\text{ for every $u$, $v\in H$},\notag\\
    &\mu(uv,w) = \mu(u,w)\mu(v,w)\text{ if $\{u,v,w\}\cap M\ne\emptyset$,}\label{Eq:WantMu}\\
    &\mu(u,vw) = \mu(u,v)\mu(u,w)\text{ if $\{u,v,w\}\cap
    M\ne\emptyset$.}\notag
\end{align}

Let $T=\{t_1,\dots,t_k\}$ be the same transversal for $M$ in $H$ as above.
Define $\mu:M\cup T\times M\cup T\to A$ as follows:
\begin{align}
    &\mu(t_1,t_1) = 1,\notag\\
    &\mu(t_i,t_i)\text{ arbitrary, for $1<i\le k$},\notag\\
    &\mu(t_i,t_j)=\delta(t_i,t_j)\text{ if $1\le i<j\le k$},\notag\\
    &\mu(t_j,t_i)=1\text{ if $1\le i<j\le k$},\label{Eq:MuParams}\\
    &\mu(m,n)=1\text{ for $m$, $n\in M$},\notag\\
    &\mu(m,t)=\delta(m,t)\text{ for $m\in M$, $t\in T$},\notag\\
    &\mu(t,m)=1\text{ for $m\in M$, $t\in T$}.\notag
\end{align}
Since $\delta(M,1) = 1$, $\mu:M\cup T\times M\cup T\to A$ is well-defined.
Extend $\mu:M\cup T\times M\cup T\to A$ to $\mu:H\times H\to A$ by
\begin{equation}
    \mu(mt,m't') = \mu(m,t')\mu(t,t'),\label{Eq:MuMM}
\end{equation}
where $m$, $m'\in M$, $t$, $t'\in T$.

\begin{lemma} $\delta(u,v) = \mu(u,v)\mu(v,u)^{-1}$ for every $u$, $v\in H$.
\end{lemma}
\begin{proof}
By the definitions \eqref{Eq:DeltaMH}, \eqref{Eq:DeltaMM} and \eqref{Eq:MuMM},
$\delta(mt,m't') = \delta(m,t')\delta(m',t)^{-1}\delta(t,t')$, and
\begin{multline*}
    \mu(mt,m't')\mu(m't',mt)^{-1}\\ = \mu(m,t')\mu(t,t')
        \mu(m',t)^{-1}\mu(t',t)^{-1} = \delta(m,t')\mu(t,t')
            \delta(m',t)^{-1}\mu(t',t)^{-1}.
\end{multline*}
Hence the desired equality holds if and only if $\mu(t,t')\mu(t',t)^{-1} =
\delta(t,t')$.

Let $t=t_i$, $t'=t_j$. If $i=j$ then $\mu(t_i,t_i)\mu(t_i,t_i)^{-1}=1 =
\delta(t_i,t_i)$. If $i<j$ then $\mu(t_i,t_j)\mu(t_j,t_i)^{-1} =
\delta(t_i,t_j)$. If $i>j$ then $\mu(t_i,t_j)\mu(t_j,t_i)^{-1} =
\delta(t_j,t_i)^{-1} = \delta(t_i,t_j)$.
\end{proof}

\begin{lemma} The following properties hold for $\mu:H\times H\to A$, $u$,
$v\in H$, and $m\in M$:
\begin{enumerate}
\item[(i)] $\mu(mu,v) = \mu(m,v)\mu(u,v)$,
\item[(ii)] $\mu(um,v) = \mu(u,v)\mu(m,v)$,
\item[(iii)] $\mu(uv,m) = \mu(u,m)\mu(v,m)$,
\item[(iv)] $\mu(u,mv) = \mu(u,m)\mu(u,v)$,
\item[(v)] $\mu(u,vm) = \mu(u,v)\mu(u,m)$,
\item[(vi)] $\mu(m,uv) = \mu(m,u)\mu(m,v)$.
\end{enumerate}
\end{lemma}
\begin{proof}
Since $\mu(m,1)=\mu(t,1) = 1$ for every $m\in M$, $t\in T$, we have
$\mu(H,M)=1$. Also, by definition \eqref{Eq:MuMM}, $\mu(mt,m't')$ does not
depend on $m'$. We will use these properties and Lemma \ref{Lm:delta2} without
reference in this proof. For (i),
\begin{align*}
    \mu(m\cdot m't',m''t'')&= \mu(mm't',t'') = \mu(mm',t'')\mu(t',t'') =
    \delta(mm',t'')\mu(t',t'')\\
    &= \delta(m,t'')\delta(m',t'')\mu(t',t'') =
    \mu(m,t'')\mu(m',t'')\mu(t',t'')\\
    &= \mu(m,t'')\mu(m't',t'') = \mu(m,m''t'')\mu(m't',m''t'').
\end{align*}
(ii) follows from (i) since $m\in Z(H)$ and $A$ is commutative. (iii) follows
from $\mu(H,M)=1$. For (iv), $\mu(u,mv)=\mu(u,v)$ and $\mu(u,m)=1$. (v) follows
from (iv) since $m\in Z(H)$ and $A$ is commutative.

For (vi), let $u=m't'$, $v=m''t''$. Then $\mu(m,uv) = \mu(m,t't'')$ and
$\mu(m,u)\mu(m,v) = \mu(m,t')\mu(m,t'')$. With $t't'' = m^*t^*$, $\mu(m,t't'')
= \mu(m,m^*t^*) = \mu(m,t^*) = \delta(m,t^*) = \delta(m,m^*)\delta(m,t^*) =
\delta(m,m^*t^*) = \delta(m,t't'') = \delta(m,t')\delta(m,t'') =
\mu(m,t')\mu(m,t'')$.
\end{proof}

\section{Explicit examples}\label{Sc:Explicit}

The minimal situation of Section \ref{Sc:NilpClassTwo} in which $f:(H/H')^3\to
A$ is nontrivial occurs when $H/H'$ is a vector space of dimension three over
$A=\{1,-1\}$, and $f$ is the unique (up to equivalence) nontrivial symmetric
trilinear alternating form, i.e., $f$ is the determinant.

The commutator subgroup $H'$ is then also of dimension three, with basis
$\{[e_1,e_2]$, $[e_1,e_3]$, $[e_2,e_3]\}$, and so $|H|=64$. One might wonder if
there is any group $H$ satisfying all these requirements. An inspection of the
GAP libraries of small groups shows that there are precisely $10$ such groups.

Furthermore, there are $21+7=28$ free parameters \eqref{Eq:DeltaParams} and
\eqref{Eq:MuParams} used in constructing $\delta$ and $\mu$ from $f$.
Altogether, when the direct product $G=A\times H$ is used, the procedure of
Section \ref{Sc:NilpClassTwo} yields $10\cdot 2^{28}$ loops (not necessarily
pairwise nonisomorphic) of order $128$ that are of nilpotency class three and
have commuting inner mappings.

Throughout this section, let $H$ be a group such that $H/H'$ is an elementary
abelian $2$-group with basis $\{e_1H',e_2H',e_3H'\}$, and $H'$ is an elementary
abelian group with basis $\{[e_1,e_2],[e_1,e_3],[e_2,e_3]\}$. Furthermore, if
$\mu:H\times H\to A$ is obtained from $f$ by the procedure of Section
\ref{Sc:NilpClassTwo}, let $\mathcal C(H,\mu)$ denote the resulting loop $Q$
defined on $G=A\times H$ via \eqref{Eq:Q}.

\subsection{The loop $C$}\label{Ss:Reconstruction}

The loop $C$ is obtained as follows: Let $H$ be the first suitable group in the
GAP library of small groups, i.e., $H$ is presented by
\begin{align*}
    H=\langle g_1, g_2, g_3, g_4, g_5, g_6;\;&
        g_i^2=1\text{ for every $1\le i\le 6$},\\
        &(g_ig_j)^2=1\text{ for every $1\le j<i\le 6$, except for}\\
        &(g_2g_1)^2g_4 = (g_3g_1)^2g_5 = (g_3g_2)^2g_6 = 1\rangle,
\end{align*}
and let all the parameters for $\delta$ and $\mu$ be equal to $1$. Then
$\mathcal C(H,\mu)$ is isomorphic to $C$.

This shows: (i) the deep insight of Cs\"org\H{o} in constructing $C$, (ii) that
the construction by group modifications is highly relevant to the problem at
hand, (iii) that $C$ is very natural among loops of nilpotency class three with
commuting inner mappings.

\subsection{About the isomorphism problem}

Different choices of $H$ and of the parameters for $\delta$ and $\mu$ produce
generally nonisomorphic loops. We do not wish to pursue the general isomorphism
problem here, but we offer some evidence that the number of loops $\mathcal
C(H,\mu)$ is very large.

\begin{lemma}\label{Lm:SmallCenter}
Let $Q=\mathcal C(H,\mu)$. Then $Z(Q) = A\times 1$.
\end{lemma}
\begin{proof}
By Lemma \ref{Lm:ZInCenters}, $A\le Z(Q)$. Now, $(a,h)\in Q$ commutes with
$(b,k)\in Q$ if and only if $(ab\mu(h,k),hk) = (ba\mu(k,h),kh)$. Thus $(a,h)\in
Z(Q)$ if and only if $h\in Z(H)=H'$ and $\delta(h,k)=1$ for every $k\in H$.

The form $f$ is determined by its values $f(e_i,e_j,e_k)$, where $1\le i$, $j$,
$k\le 3$, and we can assume without loss of generality that $f(e_i,e_j,e_k)=-1$
if and only $|\{i,j,k\}|=3$.

Suppose that $h\ne 1$. Then it is always possible to find $k$ such that
$\delta(h,k)=1$ leads to a contradiction. For instance, when
$h=[e_1,e_2][e_1,e_3]$, we let $k = e_3$, and calculate $1 = \delta(h,k) =
\delta([e_1,e_2][e_1,e_3],e_3) = \delta([e_1,e_2],e_3)\delta([e_1,e_3],e_3) =
f(e_1,e_2,e_3)f(e_1,e_3,e_3) = (-1)1 = -1$. The remaining $6$ cases are left to
the reader.
\end{proof}

\begin{lemma}\label{Lm:NotIso}
If $H_1$, $H_2$ are not isomorphic then $\mathcal C(H_1,\mu_1)$, $\mathcal
C(H_2,\mu_2)$ are not isomorphic.
\end{lemma}
\begin{proof}
Let $Q_i=\mathcal C(H_i,\mu_i)$. Assume, for a contradiction, that $Q_1\cong
Q_2$. By Lemma \ref{Lm:SmallCenter}, the two centers $Z(Q_1)$, $Z(Q_2)$ are
equal to $A\times 1$, and thus $H_1\cong Q_1/Z(Q_1)\cong Q_2/Z(Q_2)\cong H_2$,
by Lemma \ref{Lm:ZInCenters}.
\end{proof}

In order to further demonstrate the multitude of nonisomorphic loops $\mathcal
C(H,\mu)$, we conducted two experiments.

First, we let $H$ be the group of Subsection \ref{Ss:Reconstruction}, set all
parameters \eqref{Eq:MuParams} of $\mu$ to $1$, and chose $\delta$ so that
precisely one parameter of \eqref{Eq:DeltaParams} was nontrivial. It turns out
that the resulting $21$ loops are pairwise nonisomorphic.

Second, we attempted to estimate the probability that $\mathcal C(H,\mu_1)$,
$\mathcal C(H,\mu_2)$ are isomorphic, if the parameters for $\mu_1$ and $\mu_2$
are chosen at random. Let $X$ be an $n$-element set partitioned into $k$
nonempty blocks, and let $p$ be the probability that two randomly chosen
elements of $X$ belong to the same block. When $n$ and $k$ are fixed, $p$ is
minimized when all blocks have the same size $n/k$, in which case $p = 1/k$.
None of the $2500$ random pairs that we tested consisted of isomorphic loops.
We can therefore conclude with some confidence that $p<1/2500$, and,
consequently, that there are at least $2500$ pairwise nonisomorphic loops with
the desired properties. Of course, it is reasonable to expect that $k$ is much
larger. We did not check more random pairs since the test for isomorphism is
time-consuming.

\subsection{Multiplication groups and inner mapping groups of loops $\mathcal C(H,\mu)$}

The original construction of Cs\"{o}rg\H{o} is based on multiplication groups,
so we look at them more closely.

The multiplication groups of loops $\mathcal C(H,\mu)$ can have different
orders. Let $\mu_0$ denote the mapping obtained when all parameters
\eqref{Eq:DeltaParams} and \eqref{Eq:MuParams} are trivial, and $\mu_1$ the
mapping obtained when all parameters are trivial except for
$\mu_1(t_2,t_2)=-1$. Let $H$ be any of the $10$ suitable groups, and let
$Q_i=\mathcal C(H,\mu_i)$. Then $|\mlt Q_0| = 2^{13}$ and $|\mlt Q_1| =
2^{17}$. We also came across a loop $\mathcal C(H,\mu)$ with multiplication
group of order $2^{16}$.

It appears that $|\mlt \mathcal C(H,\mu)| \ge 2^{13}$, and that the equality
holds if and only if $\mu$ is trivial.

Even if their orders agree, the multiplication groups need not be isomorphic.
For instance, $\mlt C$ is not isomorphic to any $\mlt \mathcal(H^*,\mu_0)$ when
$H^*$ is a group different from $H$ of Subsection \ref{Ss:Reconstruction}.

On the other hand, the structure of the inner mapping group of $\mathcal
C(H,\mu)$ is clear:

\begin{proposition}
$\inn \mathcal C(H,\mu)$ is an elementary abelian $2$-group.
\end{proposition}
\begin{proof}
Let $Q=\mathcal C(H,\mu) = A\times H$. It suffices to show that the left,
right, and middle inner mappings of $Q$ are involutions, since we already know
that $\inn Q$ is abelian.

By \eqref{Eq:Associator}, $[x,y,z]\in Z=A$ for every $x$, $y$, $z\in Q$. By
\eqref{Eq:TwoLeftInner} and \eqref{Eq:TwoRightInner}, $L(x,y)^2z = z[x,y,z]^2 =
z$ and $R(x,y)^2z = z[z,x,y]^2 = z$, since $A$ is of exponent two.

By \eqref{Eq:TwoMiddleInner},
\begin{displaymath}
    T(x)^2z = x^{-2}zx^2\mu(z^x,x)\mu(x,z^{x^2})^{-1}\mu(z,x)\mu(x,z^x)^{-1}.
\end{displaymath}
Since $x^2\in Z(G)$, we can rewrite this as
\begin{align*}
    T(x)^2z&=z\mu(z^x,x)\mu(x,z)^{-1}\mu(z,x)\mu(x,z^x)^{-1}=z\delta(z^x,x)\delta(z,x)\\
        &= z\delta(z[z,x],x)\delta(z,x)=z\delta(z,x)^2\delta([z,x],x) = z\delta([z,x],x).
\end{align*}
As $\delta([z,x],x) = f(z,x,x) = 1$, we are done.
\end{proof}

\section{Open problems}\label{Sc:Open}

These are the main open problems of interest here:

\begin{problem} Let $Q$ be a loop of nilpotency class at least three with
abelian group of inner mappings.
\begin{enumerate}
\item[(i)] Can the nilpotency class of $Q$ be bigger than three?

\item[(ii)] Can $|Q|$ be less than $128$?

\item[(iii)] Can $|Q|$ be odd?

\item[(iv)] Can $Q$ be constructed by the modifications of Section
\emph{\ref{Sc:Modifications}}?

\item[(v)] Can $Q$ be Moufang?

\item[(vi)] Can $|\mlt Q|$ be less than $8192$?
\end{enumerate}
\end{problem}

While this paper was under review, G.~P.~Nagy and the second author constructed
a Moufang loop of nilpotency class three and with commuting inner mappings,
hence solving (v).

\end{document}